\documentclass[11pt]{article}

\usepackage{amssymb}
\usepackage{graphicx}
\usepackage{color}
\usepackage{amsmath,amsthm,amsfonts,epsfig,setspace}
\numberwithin{equation}{section}

\newtheorem{thm}{Theorem}[section]
\newtheorem{rmk}{Remark}[section]
\newtheorem{cor}{Corollary}[section]

\newtheorem{lem}{Lemma}[section]

 \def\p{\partial}
\def \Vh0{\stackrel{\circ}{V}_h}

\def\l{\label}  \def\f{\frac}  

\def\K{\texttt{K}}

\def\m{\mbox}   \def\lam{\lambda}

\def\l|{\left|}
\def\r|{\right|}

\newcommand{\R}{\mathbb{R}}

\newcommand{\lc}
{\mathrel{\raise2pt\hbox{${\mathop<\limits_{\raise1pt\hbox
{\mbox{$\sim$}}}}$}}}

\newcommand{\gc}
{\mathrel{\raise2pt\hbox{${\mathop>\limits_{\raise1pt\hbox{\mbox{$\sim$}}}}$}}}

\newcommand{\ec}
{\mathrel{\raise2pt\hbox{${\mathop=\limits_{\raise1pt\hbox{\mbox{$\sim$}}}}$}}}

\def\be{\begin{equation}} \def\ee{\end{equation}}

\def\bea{\begin{eqnarray}}  \def\eea{\end{eqnarray}}

\def\beas{\begin{eqnarray*}} \def\eeas{\end{eqnarray*}}

\def\bn{\begin{enumerate}} \def\en{\end{enumerate}}

\def\bd{\begin{description}} \def\ed{\end{description}}

\title{Super-resolution in high contrast media\thanks{\footnotesize This work was supported  by the
ERC Advanced Grant Project MULTIMOD--267184.}}
\date{}

\author{
Habib Ammari\thanks{\footnotesize Department of Mathematics and Applications,
Ecole Normale Sup\'erieure, 45 Rue d'Ulm, 75005 Paris, France
(habib.ammari@ens.fr, zh.hai84@gmail.com).} \and   Hai Zhang\footnotemark[2]
}

\begin{document}
\maketitle

\begin{abstract}
A mathematical theory is developed to explain the super-resolution and super-focusing  in high contrast media.   The approach is based on the resonance expansion of the Green function associated with the medium. 
It is shown that the super-resolution is due to sub-wavelength resonant modes excited in the medium which can propagate into the far-field.  
\end{abstract}

\medskip

\bigskip

\noindent {\footnotesize Mathematics Subject Classification
(MSC2000): 35R30, 35B30.}


% \tableofcontents

\section{Introduction} \label{sec-intro}
It is well-known that the resolution in the homogeneous space for far-field imaging system is limited by half the operating wave-length, which is a direct consequence of Abbe's diffraction limit. In order to differentiate point sources which are located less than half the wavelength apart, super-resolution techniques have to be used. While many techniques exist in practice, here we are only interested in the one using resonant media. The resolution enhancement in resonant media has been demonstrated in various recent experiments \cite{physics1, physics2,physics3, physics4, physics5}. The basic idea is the following: suppose that we have sources that are densely located in a homogeneous space of size the wavelength of the wave the sources can emit, and we want to differentiate them by making measurements in the far-field. While this is impossible in the homogeneous space, it is possible if the medium around these sources is changed so that the point spread function \cite[p.35]{book1}, which is the imaginary part of the Green function in the new medium, displays a much sharper peak than the homogeneous one and thus can resolve sub-wavelength details. The key issue in such an approach is to design the surrounding medium so that the corresponding Green function has the tailored property. 

In this paper, we develop the mathematical theory for realizing this approach by using high contrast media. We show that in high contrast media the super-resolution is due to the propagating sub-wavelength
resonant modes excited in the media and is limited by the finest structure in these modes. 
It is worth emphasizing that this mechanism  is similar to the one using Helmholtz resonators, which was recently investigated in \cite{hai, physics2}.

The paper is organized as follows. In section \ref{sec2} we recast the imaging problem as an inverse source problem and  outline different approaches for solving the inverse source problem. We emphasize that time-reversal is a direct imaging method while $L^2$- and $L^1$-minimization methods are post-imaging processes by using a prior information. In section \ref{sec3} we derive expansions of the Green function in a high contrast medium and provide a mathematical foundation for the super-resolution, which is the counterpart of super-focusing. The paper ends with a short discussion. 

%\section{Imaging methods and resolution analysis in media}

\section{Inverse source problems} \label{sec2}

We consider the following inverse source problem in a general medium characterized by refractive index $n(x)$,
\beas
&\Delta u + k^2 n(x) u = f ,\\
& u \,\m{ satisfies the Sommerfeld radiation condition}.
\eeas

We assume that $n-1$ is compactly supported in a bounded domain $D\subset \R^d$ and is assumed to be known.
We are interested in imaging $f$, which can be either a function in $L^2(D)$ or consists of finite number of point sources  supported in $D$, from the scattered field  $u$ in the 
far-field.
Denote by $G(x, y, k)$ the corresponding Green function for the media, that is, the solution to
\beas
&\Delta G(x,y,k) + k^2 n(x) G(x,y,k) = \delta(x-y), \\
& u \,\m{ satisfies the Sommerfeld radiation condition}
\eeas
with $\delta$ being the Dirac mass, we have
$$
u(x) = \K_D[f](x):=\int_D G(x, y, k)f(y)\, dy.
$$

The inverse source problem of reconstructing $f$ from $u$ for fixed frequency is well-known to be ill-posed for general sources; see, for instance, \cite{book1,book2, gang}. 
While there are many methods of reconstructing $f$ from $u$, we are interested in the following three most common ones in the literature:
\begin{enumerate}
\item
Time reversal based method;
\item
Minimum $L^2$-norm solution;
\item
Minimum $L^1$-norm solution.
\end{enumerate}

\subsection{Time reversal based method}
We first present some basics about the
time reversal based method. The imaging functional is given as follows
\be
I(x) = \int_{\Gamma} \overline{G(x, z, k)} u(z)\, ds(z) = \K^*_D \K_D[f](x),
\ee
where $\Gamma$ is a closed surface in the far-field where the measurements are taken, and $\K^*_D$ is the adjoint of $\K_D$ viewed as a linear operator from the space $L^2(D)$ to $L^2(\Gamma)$.
The physical meaning of the operator $\K^*_D$ is to time-reversing (or focusing) the observed field.
This imaging method is the simplest and perhaps the mostly used one in practice.

The resolution of this imaging method can be derived from the following
Helmholtz-Kirchhoff identity:
\be
\int_{\Gamma} \Big( \overline{G(x, z, k)} \f{\p G(y, z, k)}{\p \nu}
- \f{\p \overline{G(x, z, k)}}{\p \nu} G(y, z, k)   \Big)  \, ds(z)= - 2 i \Im{G}(x,y, k), \quad \forall x, y \in D.
\ee

Note that in the far-field, we can use the Sommerfeld radiation condition, as a result, we obtain the following lemma. 

\begin{lem} Let $G$ be the Green function and let $\Gamma$ be a smooth closed surface. We have
\be
k \int_{\Gamma} \overline{G(x, z, k)} G(y, z, k) \, ds(z) = - \Im{G}(x,y, k) + O(\f{1}{R}),
\ee
where $R$ is the distance between the far-field surface $\Gamma$, where the measurements are taken, and $D$, where the sources are located.
\end{lem}

As a corollary, the following result holds.
\begin{cor} We have
$$
I(x) = \K^*_D \K_D[f](x) \approx  - \f{1}{k} \int_D \Im{G}(x,y, k) f(y)\, dy.
$$
\end{cor}

If we take $f$ to be a point source, we obtain the point spread function of the imaging functional, which shows that the time-reversal based method has resolution limited by $\Im{G}(x,y, k)$.

\subsection{Minimum $L^2$-norm solution}
We now consider the second method which is based on $L^2$-minimization.
We assume that the source $f\in L^2(D)$.
The method is given as follows
\begin{equation} \label{l2m}
\min {\|g\|_{L^2(D)}} \,\mbox{ subject to } \,\,\,\, K_D[g] =u,\end{equation}
which can be relaxed in the presence of noise as follows
\begin{equation} \label{l2mr}
\min {\|g\|_{L^2(D)}} \,\mbox{ subject to } \,\,\,\, \|\K_D[g] -u\|^2_{L^2(\Gamma)} < 
\delta
\end{equation}
with $\delta >0$ being a known small parameter.

In order to obtain an explicit formula for this method, we consider the singular value decomposition for the operator
$$
\K_D : L^2(D) \rightarrow  L^2(\Gamma).
$$
We have
$$
\K_D = \sum_{l\geq 0} \sigma_{l} P_l,
$$
where $\sigma_l$ is the l-th singular value and $P_{l}$ is the associated projection.
The ill-posedness of the inverse source problem is due to the fast decay of the singular values to zero; see, for instance, \cite{book2, slepian}. 

By a direct calculation, one can show that the minimum $L^2$-norm solution to (\ref{l2m}) is given by
\be
I(x) = \sum_{l\geq 0} \f{P_l^*P_l}{\sigma_{l}^2} \K_D^* \K_D[f](x),
\ee
while the regularized one, which is the solution to (\ref{l2mr}) is given by
\be
I_\alpha(x) = \sum_{l\geq 0} \f{P_l^*P_l}{\sigma_{l}^2 + \alpha} \K_D^* \K_D[f](x),
\ee
with $\alpha$ as a function of $\delta$ being chosen by Morozov's discrepancy principle; see, for instance, \cite{otmar}.

\subsection{Minimum $L^1$-norm solution}
The method of minimum $L^1$-norm solution is proposed by Cand\`es and Fernandez-Granda in the recent papers \cite{candes1, candes2}.
The authors assume that $f$ is equal to suppositions of separate point sources.
Their method is to solve the minimization problem
$$
\min {\|g\|_{L^1(D)}} \, \,  \mbox{ subject to }\,\,   \K_D^*\K_D[g] = \K_D^*[u],
$$
or its relaxed version, which reads as
$$
\min {\|g\|_{L^1(D)}} \, \, \mbox{ subject to
 }\,\, \|\K_D^*\K_D[g] -\K_D^*[u] \|^2_{L^2(\Gamma)}  < \delta.
$$
They show that under a minimum separation condition for the point sources, the inverse source problem is well-posed.
A main feature of their approach is that the $L^1$-minimization can pull out small spikes even though they may be completely buried in the side lobes of the large ones. 
%which means that both uniquely and stability can be guarantied.

\subsection{The special case of homogeneous medium}

In homogeneous medium, we have $n \equiv 1$. For simplicity, we consider the case $d=3$.
$$
G(x, y, k) = G_0(x, y, k) = - \f{e^{ik|x-y|}}{4 \pi |x-y|}.
$$

In the far-field, where $k|y|= O(1)$ and $k|x| >> 1 $, we have $|x-y| \approx  |x| - \hat{x}\cdot y$, where $\hat{x} = \f{x}{|x|}$.
Thus,
$$
u(x) = - \int_{D}  \f{e^{ik|x-y|}}{4 \pi |x-y|} f(y)  \, dy \approx -  \f{e^{ik|x|}}{4 \pi |x|} \hat{f}(k\hat{x}),
$$
where $\hat{f}$ is the Fourier transform of $f$. 

If we make measurements on the surface $\Gamma = S(0, R)$, the sphere of radius $R$ and center the origin,  then we have
$$
u(x)= - \f{e^{ikR}}{4 \pi R} \hat{f}(k\hat{x}).
$$

Using the time-reversal method, we have
$$
I(z) = - \f{1}{16k \pi^2 R^2} \int_{|y|= R}  e^{ik \hat{y}\cdot(x-z)}  \hat{f}(k\hat{y}) \, ds(y)= -
\f{1}{4\pi } \int_{D}  f(y) \frac{\sin k |z-y|}{k |z-y|} \, dy.
$$

\section{The Green function in high contrast media} \label{sec3}

Throughout this section, we put the frequency $k$ to be the unit and suppress its 
presence in what follows. We consider the following Helmholtz equation with a delta source
term 
\begin{align}
 \Delta_x G(x,x_0) + G(x, x_0) + \tau n(x)\chi_{D}G(x,x_0) = \delta(x-x_0) \quad  \mbox{in} \,\, \R^d,
\end{align}
where $\chi_{D}$ is the characteristic function of a bounded domain $D\subset \R^d$, $n(x)$ is a positive function of order one in the space of  $C^1(\bar{D})$ and $\tau \gg 1$ is the contrast. We denote by
$G_0(x, x_0)$ the free space Green function.

Write $G= v+G_0$, we can show that
\be
\Delta v + v = -\tau n(x) \chi_D (v+G_0).
\ee
Thus,
$$
v(x, x_0)= -\tau\int_{D} n(y) G_0(x, y) \bigg(v(y, x_0)+ G_0(y, x_0)\bigg)\, dy .
$$
Define
\be
\K_D[f](x) = -\int_{D} n(x) G_0(x, y)f(y)\, dy.
\ee
Then, $v=v(x)= v(x, x_0)$ satisfies the following integral equation
\be
(I -\tau\K_D) [v] = \tau\K_D [G(\cdot, x_0)],
\ee
and hence, 
$$
v(x) = (\f{1}{\tau} - \K_{D})^{-1} \K_D [G(\cdot, x_0)].
$$

In what follows, we present properties of the integral operator $\K_D$.

\begin{lem}
The operator $\K_D$ is compact from $L^2(D)$ to $L^2(D)$. In fact, $\K_D$ is bounded from $L^2(D)$ to $H^2(D)$. Moreover,
$\K_D$ is a Hilbert-Schmidt operator.
\end{lem}

\begin{lem} \label{lem-11} Let $\sigma(\K_D)$ be the spectrum of $\K_D$. We have
\begin{enumerate}
\item[(i)]
$\sigma(\K_D) = \{0, \lambda_1, \lambda_2, ..., \lambda_n, ....\}$, where $|\lambda_1| \geq |\lambda_2| \geq |\lambda_3|\geq...  $ and $\lambda_n \rightarrow 0$;

\item[(ii)]
$\{0\} = \sigma(\K_D) \backslash  \sigma_p (\K_D)$ with $\sigma_p (\K_D)$ being the point spectrum of $\K_D$.

\end{enumerate}
\end{lem}

\begin{proof}
We need only prove the second assertion. Assume that $\K_D [u] =\int_{D} G_0(x, y)n(y)u(y) \, dy =0$.
We have
$0= (\triangle+1 ) \K_D [u] = n u$, which shows that $u=0$. The assertion is then proved.   
\end{proof}

\begin{lem} \label{lem-eigen}
$\lambda \in \sigma(\K_D)$ if and only if there is a nontrivial solution in $H^2_{\mathrm{loc}}(\R^d)$ to the following problem
\begin{eqnarray}
 &(\Delta +1) u(x) = \f{1}{\lambda} n(x)u(x) \quad \mbox{in } \,D , \label{super-ocillatory-mode1}\\
  &(\Delta +1) u= 0 \quad   \mbox{in } \,\R^d\backslash D, \label{super-ocillatory-mode2}\\
  & u \m{ satisfies the Sommerfeld radiation condition}.
 \end{eqnarray}
\end{lem}

\begin{proof}
 Assume that $\K_D[u] = \lam u$. We define $\tilde{u}(x)= \int_{D} G(x, y)n(y)u(y) \, dy$, where $x\in \R^d$. Then $\tilde{u}$  satisfies the required equations.
\end{proof}

We call functions satisfying (\ref{super-ocillatory-mode1})- (\ref{super-ocillatory-mode2}) the resonant modes. They have sub-wavelength structures in $D$ for $|\lam| < 1$ and can propagate into the far-field. It is these sub-wavelength propagating modes that causes super-resolution. We may also call them super-oscillatory modes.

\begin{rmk}
It is clear that $\lambda$ is a non-zero real eigenvalue for the operator $\K_D$ if and only if $1$ is a transmission eigenvalue for the medium characterized by $1-\f{1}{\lambda} n(x)$, i.e., there exists a nontrivial solution to the following equations
\begin{eqnarray*}
& (\Delta +1-\f{1}{\lambda} n(x) ) u(x) = 0  \quad \mbox{in } \,D, \label{super-ocillatory-mode1b}\\
&  (\Delta +1) u= 0 \quad   \mbox{in } \,\R^d\backslash D, \label{super-ocillatory-mode2b}\\
& u \in H^2_{\mathrm{loc}}(\R^d)\\
  & u \m{ satisfies the Sommerfeld radiation condition}.
\end{eqnarray*}
\end{rmk}
We refer to \cite{sylvester} for a discussion on transmission eigenvalue problems.

We denote by $\mathcal{H}_{j}$ the generalized eigenspace of the operator $\K_D$ for the eigenvalue $\lam_j$.

\begin{lem} The following decomposition holds:
$$L^2(D) = \overline{ \bigcup_{j=1}^{\infty}\mathcal{H}_{j} }.$$ 
\end{lem}

\begin{proof}
By the similar method as in the proof of Lemma \ref{lem-11}, we can show that $ \m{Ker } \K_D^* = \{0\}$.
As a result, we have
$$
 \overline{\K_D \big(L^2(D)\big)} =  \big(\m{Ker } \K_D^*\big)^{\perp} = L^2(D).
$$ 
The lemma is proved.  
\end{proof}

\begin{lem}
There exists a basis $\{u_{j, l, k}\}$, $1\leq l \leq m_j, 1\leq k \leq n_{j, l}$ for  $\mathcal{H}_{j}$ such that
\begin{equation*}
\K_D(u_{j, 1, 1}, ... , u_{j, m_j, n_{j, m_j}}) = (u_{j, 1, 1}, ... , u_{j, m_j, n_{j, m_j}}) \begin{pmatrix}
 J_{j,1} &  &  \\  & \ddots &  \\  &  & J_{j, m_j} 
 \end{pmatrix},
\end{equation*}
where $J_{j, l}$ is the canonical Jordan matrix of size $n_{j, l}$ in the form
\begin{equation*}
J_{j,l} = \begin{pmatrix}
 \lam_j & 1 &  &  \\ & \ddots & \ddots &   \\&  & \lambda_j & 1 \\ & & & \lam_j
 \end{pmatrix}.
\end{equation*}
\end{lem}

\begin{proof}
This follows from the Jordan theory applied to the linear operator $\K_D|_{\mathcal{H}_{j}}: \mathcal{H}_{j} \rightarrow \mathcal{H}_{j}$ on the finite dimensional space $\mathcal{H}_{j}$.
\end{proof}

We denote
$\Gamma= \{(j, l, k) \in \textbf{N} \times \textbf{N} \times \textbf{N}; 1\leq l \leq m_j, 1\leq k \leq n_{j, l}\}$ the set of indices for the basis functions.
%$\gamma= (j, l, k) \in \textbf{N} \times \textbf{N} \times \textbf{N}$.
We introduce a partial order on $\textbf{N} \times \textbf{N} \times \textbf{N}$.
Let $\gamma=(j, k, l) \in \Gamma, \gamma'= (j', l', k') \in \Gamma$, we say that
$\gamma' \preceq \gamma$ if one of the following conditions are satisfied:
\begin{enumerate}
\item[(i)]
$j>j'$;
\item[(ii)]
$j=j'$, $l>l'$;
\item[(iii)]
  $j=j'$, $l=l'$, $k\geq k'$.
\end{enumerate}

%
%Note that $\mathcal{H}_{\Delta^c}$ is an invariant subspace for the operator $\K_D$. We define $\K_{\Delta^c} = K_{D}|_{\mathcal{H}_{\Delta^c}}$.

By Gram-Schmidt orthonormalization process, the following result is obvious.
\begin{lem}
There exists orthonormal basis $\{e_{\gamma}: \gamma\in \Gamma\}$ for $\mathcal{H}$
such that
$$
e_{\gamma} = \sum_{\gamma' \preceq \gamma} a_{\gamma, \gamma'} u_{\gamma'},
$$
where $a_{\gamma, \gamma'}$ are constants and $a_{\gamma, \gamma} \neq 0$.
\end{lem}

We can regard $A= \{a_{\gamma, \gamma'}\}_{\gamma, \gamma' \in \Gamma}$ as a matrix. It is clear that $A$ is upper-triangular and has non-zero diagonal elements. Its inverse is denoted by $B= \{b_{\gamma, \gamma'}\}_{\gamma, \gamma' \in \Gamma}$ which is also upper-triangular and has non-zero diagonal elements. We have
$$
u_{\gamma}= \sum_{\gamma' \preceq \gamma} b_{\gamma, \gamma'} e_{\gamma'}.
$$

\begin{lem}
The functions $\{e_{\gamma}(x)\overline{e_{\gamma'}(y)}\}$ form a normal basis for the Hilbert space $L^2(D\times D)$.
Moreover, the following completeness relation holds: 
$$
\delta(x-y) = \sum_{\gamma} e_{\gamma}(x)\overline{e_{\gamma}(y)}.
$$
\end{lem}

By standard elliptic theory, we have
$G(x, x_0) \in L^2(D\times D)$ for fixed $k$. Thus we have
\be
G(x, x_0)= \alpha_{\gamma, \gamma'} e_{\gamma}(x) \overline{e_{\gamma}}(x_0),
\ee
for some constants
$\alpha_{\gamma, \gamma'}$ satisfying
$$
\sum_{\gamma, \gamma' } |\alpha_{\gamma, \gamma'}|^2 = \|G(x, x_0)\|^2_{L^2(D\times D)} < \infty.
$$

To analyze the Green function $G$, we need to find the constants $\alpha_{\gamma, \gamma'}$. For doing so, we first note that
$$
G_0(x, x_0) = \f{1}{n(x_0)} \K_D [\delta( \cdot-x_0)].
$$
Thus,
\beas
G(x, x_0) &= &  G_0(x, x_0) + (\f{1}{\tau} - \K_{D})^{-1} \K_D ^2 [\delta(\cdot - x_0)] \\
&= &  G_0(x, x_0) + \f{1}{n(x_0)} \sum_{\gamma} \overline{e_{\gamma}}(x_0) (\f{1}{\tau} - \K_{D})^{-1} \K_D ^2 [e_{\gamma}].
\eeas

We next compute $(\f{1}{\tau} - \K_{D})^{-1} \K_D ^2 [e_{\gamma}]$.
For ease of notation, we define $u_{j, l, k} =0$ for $k\leq 0$.
We have
\beas
\K_D [u_{j, l, k}] &=& \lam_j u_{j, l, k} + u_{j, l, k-1} \quad \m{for all} \,\, j, l, k,\\
\K_D^2 [u_{j, l, k}] &=& \lam_j^2 u_{j, l, k} + 2\lam_j u_{j, l, k-1}+u_{j, l, k-2}  \quad \m{for all} \,\, j, l, k.
\eeas

On the other hand, for $z \notin \sigma(\K_D)$, we have
\beas
(z- \K_D) ^{-1} [u_{j, l, k}] = \f{1}{z-\lam_j} u_{j, l, k} + \f{1}{(z-\lam_j)^2} u_{j, l, k-1}
+... + \f{1}{(z-\lam_j)^{k}}u_{j, l, 1},
\eeas
and therefore, it follows that
\beas
(z- \K_D) ^{-1} \K_D^2  [u_{j, l, k}]
&=&  \f{\lam_j^2}{z-\lam_j} u_{j, l, k}+ \f{\lam_j^2}{(z-\lam_j)^2} u_{j, l, k-1} \dotsb +
  \f{ \lam_j^2}{(z-\lam_j)^{k}} u_{j, l, 1} \\
&& + \f{2 \lam_j}{z-\lam_j} u_{j, l, k-1}+ \f{2\lam_j}{(z-\lam_j)^2} u_{j, l, k-2} \dotsb +
  \f{ 2 \lam_j}{(z-\lam_j)^{k-1}} u_{j, l, 1} \\
&& + \f{1}{z-\lam_j} u_{j, l, k-2}+ \f{1}{(z-\lam_j)^2} u_{j, l, k-3} \dotsb +
  \f{1}{(z-\lam_j)^{k-2}} u_{j, l, 1}\\
&=&    \f{\lam_j^2}{z-\lam_j} u_{j, l, k} + \Big(\f{\lam_j^2}{(z-\lam_j)^2}+ \f{2 \lam_j}{z-\lam_j}\Big) u_{j, l, k-1}\\
&& + \Big(\f{\lam_j^2}{(z-\lam_j)^3}+ \f{2 \lam_j}{z-\lam_j}+ \f{1}{z-\lam_j} \Big) u_{j, l, k-2}\\
&& + ... +  \Big(\f{\lam_j^2}{(z-\lam_j)^k}+ \f{2 \lam_j}{(z-\lam_j)^{k-1}}+ \f{1}{(z-\lam_j)^{k-2}} \Big) u_{j, l, 1} \\
&=& \sum_{\gamma'} d_{\gamma, \gamma'} e_{\gamma'},
\eeas
where we have introduced the matrix $D= \{d_{\gamma, \gamma'}\}_{\gamma, \gamma' \in \Gamma}$, which is upper-triangular and has block-structure.

%We define
%\begin{equation}
%
%d_{\gamma, \gamma^{'}} = \begin{cases} \f{\lam_j^2}{z-\lam_j}, & \mbox{if } \gamma= \gamma' \\
%, & \mbox{if } n\mbox{ is odd} \end{cases}
%
%\end{equation}

With these calculations, by taking $z=1/\tau$, we arrive at the following result. 

\begin{thm}
The following expansion holds for the Green function
\be
G(x, x_0) = G_0(x, x_0) + \sum_{\gamma \in \Gamma} \sum_{\gamma' \preceq \gamma} \alpha_{\gamma, \gamma'} e_{\gamma}(x) \overline{e_{\gamma'}}(x_0),
\ee
where
$$
\alpha_{\gamma, \gamma^{'''}} = \f{1}{n(x_0)}\sum_{\gamma' \preceq \gamma} \sum_{\gamma^{''} \preceq \gamma}
\sum_{\gamma^{'''} \preceq \gamma^{''}}
 a_{\gamma, \gamma'}d_{\gamma', \gamma^{''}}b_{\gamma^{''}, \gamma^{'''}}.
$$

Moreover, for $\tau$ belonging to a compact subset of $\R \setminus \big(\R \cap \sigma(\K_D) \big)$, we have the following uniform bound
$$
\sum_{\gamma, \gamma' } |\alpha_{\gamma, \gamma'}|^2 < \infty.
$$
\end{thm}

Alternatively, if we start from the identity
\beas
\delta(x-y) &= & \sum_{\gamma} e_{\gamma}(x)\overline{e_{\gamma}(y)} \\
&= & \sum_{\gamma} \sum_{\gamma'} \sum_{\gamma''}
a_{\gamma, \gamma'} u_{\gamma'}(x) \overline{a_{\gamma, \gamma''}} \overline{u_{\gamma''}}(x_0) \\
&= & \sum_{\gamma} \sum_{\gamma'} \sum_{\gamma''}
a_{\gamma, \gamma'} \overline{a_{\gamma, \gamma''}} u_{\gamma'}(x) \overline{u_{\gamma''}}(x_0), 
\eeas
then we can obtain an equivalent expansion for the Green function in terms of the basis of resonant modes.

\begin{thm}
The following expansion holds for the Green function
\be \label{exps1}
G(x, x_0) = G_0(x, x_0) + \sum_{\gamma \in \Gamma} \sum_{\gamma^{'''} \preceq \gamma} \beta_{\gamma, \gamma^{'''}} u_{\gamma}(x) \overline{u_{\gamma^{'''}}}(x_0),
\ee
where
$$
\beta_{\gamma, \gamma^{'''}} =\f{1}{n(x_0)}\sum_{\gamma^{'} \preceq \gamma} \sum_{\gamma^{''} \preceq \gamma}
\sum_{\gamma^{'''} \preceq \gamma^{'}}
 a_{\gamma, \gamma'}\overline{a_{\gamma, \gamma^{''}}} d_{\gamma^{'}, \gamma^{'''}}.
$$
Here, the infinite summation can be interpreted as follows
\be \label{exps2}
\lim_{\gamma_0 \rightarrow \infty} \sum_{\gamma \leq \gamma_0} \sum_{\gamma' \preceq \gamma} \beta_{\gamma, \gamma'} u_{\gamma}(x) \overline{u_{\gamma'}}(x_0) = G(x, x_0) -  G_0(x, x_0) \quad \mbox{in }\,\, L^2(D\times D).
\ee
\end{thm}

In order to have some idea of the expansions of the Green function $G(x, y)$, we compare them to the expansion of the Green function in the homogeneous space, i.e., $G_0(x, y)$. For this purpose, we introduce the matrix $H=\{h_{\gamma, \gamma'}\}_{\gamma, \gamma' \in \Gamma}$, which is defined by
$$
\K_{D} [u_{\gamma}] = \sum_{\gamma'} h_{\gamma, \gamma'} u_{\gamma'}.
$$
%It is clear that
%\be
%h_{\gamma, \gamma'}

In fact, we have 
$$
h_{j,l,k,j',l',k'} = \lambda_j \delta_{j,j'} \delta_{l,l'} \delta_{k,k'} + \delta_{j,j'} \delta_{l,l'} \delta_{k-1, k'}, 
$$
where $\delta$ denotes the Kronecker symbol. 

\begin{lem}
\begin{enumerate}
\item[(i)]
In the normal basis $\{e_{\gamma}\}_{\gamma\in \Gamma}$, the following expansion holds for the Green function $G_0(x, x_0)$
\be
G_0(x, x_0) = \sum_{\gamma \in \Gamma} \sum_{\gamma' \preceq \gamma} \tilde{\alpha}_{\gamma, \gamma'} e_{\gamma}(x) \overline{e_{\gamma'}}(x_0),
\ee
where
$$
\tilde{\alpha}_{\gamma, \gamma^{'''}} = \f{1}{n(x_0)}\sum_{\gamma' \preceq \gamma} \sum_{\gamma^{''} \preceq \gamma}
\sum_{\gamma^{'''} \preceq \gamma^{''}}
 a_{\gamma, \gamma'}h_{\gamma', \gamma^{''}}b_{\gamma^{''}, \gamma^{'''}}.
$$

Moreover, for $k$ belonging to a compact subset of $\R \setminus \big(\R \cap \sigma(\K_D) \big)$, we have the following uniform bound
$$
\sum_{\gamma, \gamma' } |\tilde{\alpha}_{\gamma, \gamma'}|^2 < C < \infty.
$$

\item[(ii)]
In the basis of resonant modes $\{u_{\gamma}\}_{\gamma\in \Gamma}$, the following expansion holds for the Green function $G_0(x, x_0)$
\be
G_0(x, x_0) = \sum_{\gamma \in \Gamma} \sum_{\gamma^{'''} \preceq \gamma} \tilde{\beta}_{\gamma, \gamma^{'''}} u_{\gamma}(x) \overline{u_{\gamma^{'''}}}(x_0),
\ee
where
$$
\tilde{\beta}_{\gamma, \gamma^{'''}} =\f{1}{n(x_0)}\sum_{\gamma^{'} \preceq \gamma} \sum_{\gamma^{''} \preceq \gamma}
\sum_{\gamma^{'''} \preceq \gamma^{'}}
 a_{\gamma, \gamma'}\overline{a_{\gamma, \gamma^{''}}} h_{\gamma^{'}, \gamma^{'''}}.
$$
Here, the infinite summation can be interpreted as follows
$$
\lim_{\gamma_0 \rightarrow \infty} \sum_{\gamma \leq \gamma_0} \sum_{\gamma' \preceq \gamma} \tilde{\beta}_{\gamma, \gamma'} u_{\gamma}(x) \overline{u_{\gamma'}}(x_0) = G_0(x, x_0) \quad \mbox{in }\,\, L^2(D\times D).
$$

\end{enumerate}
\end{lem}

By comparing the coefficients $\alpha_{\gamma, \gamma'}$ (or $\beta_{\gamma, \gamma'}$) and $\tilde{\alpha}_{\gamma, \gamma'}$ (or $\tilde{\beta}_{\gamma, \gamma'})$, we can see that the imaginary part of $G(x, y)$ may have sharper peak than $G_0(x, y)$ due to the excited high frequency resonant modes.

Finally, note that in the special case where the spaces $\mathcal{H}_j$ are of dimension one, we have
$$d_{\gamma, \gamma'} = \delta_{\gamma, \gamma'} \frac{\lambda_j^2}{z-\lambda_j}, \quad 
h_{\gamma,\gamma'} = \delta_{\gamma,\gamma'} \lambda_j.$$

\section{Concluding remarks}

In this paper, we provided a mathematical theory  to explain 
the super-resolution and super-focusing mechanisms in high contrast media. 

From the  expansions (\ref{exps1}) and (\ref{exps2}), we proved that the super-resolution is due to propagating sub-wavelength resonant modes. It is worth mentioning that in (\ref{exps1}) and (\ref{exps2}), we observed that a phenomenon of mixing of modes occurs. This is essentially due to the non-hermitian nature of the system (the operator $\K_D$) we considered. We believe that the  mixing of resonant modes is an intrinsic nature of non-hermitian systems, as opposed the the eigen-expansion for hermitian systems, because of the rigorous mathematical convergence results. However, this phenomenon is sometimes ignored in physics literature where formal resonance expansions without mixing are proposed without any evidence of convergence.

Our approach in this paper for inverse source problems complements the one recently proposed in \cite{scatcoef}, which is based on the concept of scattering coefficients and solves the superresolution problem for inverse scattering problems. 

Finally, we expect that our present approach could provide a mathematical  explanation of the  mechanism of super-resolution and super-focusing in other resonant media including negative index materials.

\end{document}